\documentclass[reqno]{amsart}
\usepackage{hyperref}
\usepackage{amsmath,amsthm,amssymb}
\usepackage{mathrsfs}
\usepackage{bbm}   
\usepackage{graphicx}

\usepackage{paralist}

\usepackage[all]{xy}
\input xy
\xyoption{all}
\xyoption{arc}
\xyoption{line}
\CompileMatrices

\newcommand{\beq}{\begin{displaymath}}
\newcommand{\eeq}{\end{displaymath}}
\newcommand{\beqn}{\begin{equation}}
\newcommand{\eeqn}{\end{equation}}
\newcommand{\beqa}{\begin{eqnarray*}}
\newcommand{\eeqa}{\end{eqnarray*}}
\newcommand{\beqna}{\begin{eqnarray}}
\newcommand{\eeqna}{\end{eqnarray}}

\newcommand{\re}[1]{~(\ref{#1})}

\newcommand{\N}{\mathbb{N}}

\newcommand{\R}{\mathbb{R}}
\newcommand{\C}{\mathbb{C}}
\renewcommand{\P}{\mathbb{P}}
\newcommand{\id}{\mathrm{id}}
\newcommand{\ra}{\rightarrow}

\newcommand{\lra}{\longrightarrow}

\numberwithin{equation}{section}

\newtheorem{prop}{Proposition}[section]
\newtheorem{thm}[prop]{Theorem}
\newtheorem{defn}[prop]{Definition}

\newtheorem{lem}[prop]{Lemma}

\theoremstyle{definition} 
\newtheorem{ex}[prop]{Example}
\newtheorem{rem}[prop]{Remark}

\numberwithin{equation}{section}

\newcommand{\ind}[2]{\emph{#1}}

\title{Convex Spaces I: Definition and Examples}
\author{Tobias Fritz}
\email{tfritz@perimeterinstitute.ca}
\address{Max Planck Institute for Mathematics, Vivatsgasse 7, 53111 Bonn, Germany.}
\thanks{I would like to thank the Max Planck Institute for providing an excellent research environment and finanical support. Branimir \'Ca\'ci\'c and Jens Putzka provided helpful comments on a previous version of this paper. Marc Olschok, George Svetlichny and Klaus Keimel have kindly pointed out relevant literature.}

\begin{document}

\maketitle

\begin{abstract}
We propose an abstract definition of convex spaces as sets where one can take convex combinations in a consistent way. A priori, a convex space is an algebra over a finitary version of the Giry monad. We identify the corresponding Lawvere theory as the category from~\cite{pcsm} and use the results obtained there to extract a concrete definition of convex space in terms of a family of binary operations satisfying certain compatibility conditions. After giving an extensive list of examples of convex sets as they appear throughout mathematics and theoretical physics, we find that there also exist convex spaces that cannot be embedded into a vector space: semilattices are a class of examples of purely combinatorial type. In an information-theoretic interpretation, convex subsets of vector spaces are probabilistic, while semilattices are possibilistic. Convex spaces unify these two concepts.
\end{abstract}

\medskip

\tableofcontents

\emph{\textbf{Important Note}: The idea of abstract convexity is not original to this article, but has been rediscovered many times over. The original version of this manuscript, having appeared as a preprint in 2009, constitutes another one of these rediscoveries that was made in 2008 with intended application to an operational axiomatization of quantum mechanics.}

\emph{In citations, please refer to the original literature. Most importantly, this literature includes the following original works:}

\emph{\begin{itemize}
\item The idea of abstract convexity seems to originate with Stone~\cite{Stone}, whose \emph{barycentric calculus} axiomatizes convex subsets of vector spaces. Stone's axioms are very similar to our Definition~\ref{convspcdefn} together with a cancellation axiom.
\item Close to the categorical approach presented here is the paper of Neumann~\cite{Neumann}, where convex spaces are investigated from the perspective of universal algebra. Among other things, Neumann also describes the distinction between convex spaces satisfying a cancellation axiom (our \emph{geometric type}) and semilattices (our \emph{combinatorial type}), and also notes the existence of convex spaces of mixed type.
\item Our intended application of convex spaces was in an operational approach to the foundations of quantum mechanics. As it turns out, this has first been done by Gudder~\cite{Gudder1}.
\item The work of \'Swirszcz~\cite{Swi1,Swi2} develops an approach based on investigating categories of convex sets, the (non-)monadicity of the associated forgetful functors, and the algebras of the resulting monads. This completely subsumes the entire development of our Section~\ref{def}, including Definition~\ref{convspcdefn}, and also contains lots of results on categorical aspects of convex sets in a functional analytic context.
\end{itemize}
}

\bigskip

\emph{In conclusion, the current manuscript should be considered as secondary literature without original contributions. Its most useful aspect may be the collection of examples in Sections~\ref{examplesgeom} to~\ref{examplesmixed}.}

\bigskip

\section{Introduction}

Looking at the history of mathematics, one easily finds an abundance of cases where abstract generalizations of concrete structures into abstract concepts spurred a variety of interesting developments or even opened up completely new fields. Some of the most obvious examples that spring to mind are:\\

\begin{itemize}
\item The concept of a \emph{group}, which provides an abstract framework for the study of \emph{symmetries}.
\item \emph{Riemannian manifolds}, were modelled after submanifolds of $\R^n$ with their \emph{intrinsic geometry}.
\item \emph{Category theory}, conceived as an abstract framework for \emph{cohomology theories}.
\item \emph{Operators on Hilbert space}, which generalize the \emph{Fourier transform} and \emph{integral equations}.
\end{itemize}

We now consider the notion of \emph{convexity} as that property of a subset of a vector space that means that the set contains the line segment connecting every two points in that subset. Perhaps surprisingly, an abstract generalization has not (yet) been proposed for this concept of convexity. To the author's knowledge, the present literature does not contain any concept of abstract convex set that provides a nice notion of convex combinations for its elements. The aim of this paper is to remedy this omission. We note however that ideas similar to the ones presented here also appeared in the online discussion~\cite{Lei}, at about the same time as the present work started to take shape.

We shall call a set together with a certain notion of abstract convex combinations a \emph{convex space}. The most obvious examples are convex subsets of vector spaces. However, there is an entirely different class of convex spaces all of which are of a discrete nature, namely meet-semilattices, where the meet operation serves as a convex combination operation. Moreover, one can also construct examples of mixed type, where one has a semilattice as an underlying discrete structure, together with a convex subset of a vector space over each element of the semilattice. This is similar to how one can project a polytope onto its face lattice by mapping each point to the face it generates: then, the polytope becomes a ``fiber bundle'' over its face lattice with the face interiors as fibers. We describe a variant of this construction in~\cite{propclass} and show that every convex space is of such a form.

Our main motivation for studying this subject comes from quantum mechanics, in particular the search for the most general framework for theories of physics. Without loss of generality, we can assume a theory of physics to be of epistemological nature; this means that what we describe is not the actual reality of the system itself, but merely the information an observer has about the system. Now information is usually incomplete, in which case the state that the observer believes the system to be in is given by a statistical ensemble. Therefore, it seems reasonable to assume that the set of the information states has the mathematical structure of convex combinations, which correspond to statistical superpositions of ensembles. This is the framework known as \emph{general probabilistic theories}~\cite{Bar}, where the set of information states is taken to be a convex subset of a vector space. However since the underlying vector space lacks any physical motivation and solely serves the purpose of defining the convex combinations, we felt the need to develop an abstract concept of convex spaces.

We now give an outline of the paper. After settling notation in section~\ref{notation}, we start section~\ref{def} by proposing our definition of convex spaces in terms of a family of binary operations satisfying certain compatibility conditions. Using concepts from category theory, we then show that these compatibility conditions imply all the relations that we expect convex combinations to have. The main step relies on the results of~\cite{pcsm}. As a first exercise in the theory of convex spaces, we then show in theorem~\ref{convspc} how a convex space structure on a set is uniquely determined by the collection of those maps that preserve convex combinations.

The remaining three sections are entirely dedicated to various classes of examples. Section~\ref{examplesgeom} then proceeds by giving a list of examples of ``geometric type'', which refers to those convex spaces that can be written as a convex subset of a vector space. Then in section~\ref{examplescomb}, we study a discrete class of convex spaces. A discrete convex space in that sense turns out to be the same thing as a semilattice. None of these can be embedded into a vector space. Finally, section~\ref{examplesmixed} describes constructions of convex spaces that have both a geometric and a combinatorial flavor. This concludes the paper. We hope that the long list of examples explains why we deem convex spaces worthy of study.

\section{Notation}
\label{notation}
The $\texttt{typewriter font}$ denotes a category, for example $\mathtt{Set}$. As in~\cite{pcsm}, we write $[n]$ as shorthand for the $n$-element set $\{1,\ldots,n\}$. The symbol $\ast$ stands for any one-element set and also for the unique convex space over that set. For a real number $\alpha\in[0,1]$, we set $\overline{\alpha}\equiv 1-\alpha$. This notation increases readability in formulas involving binary convex combinations. The $\overline{\,\cdot\,}$ operation satisfies the important relations
\beq
\overline{\overline{\alpha}}=\alpha,\quad \overline{\alpha+\beta}=\overline{\alpha}+\overline{\beta}-1,\quad \overline{\alpha\beta}=\overline{\alpha}+\overline{\beta}-\overline{\alpha}\overline{\beta}.
\eeq
Given a set $X\in\mathtt{Set}$, we call
\beq
\Delta_X\equiv\left\{f:X \ra [0,1]\:\Bigg|\:f\textrm{ has finite support and }\sum_{x\in X}f(x)=1\right\}
\eeq
the \emph{simplex over $X$}. We also consider $\Delta_X$ as the set of all finite formal convex combinations $\sum_i\lambda_i\underline{x}_i$ with $x_i\in X$, where we use the underline notation $\underline{x}_i$ to emphasize that the sum is formal; this allows us to distinguish $x\in X$ from $\underline{x}\in\Delta_X$. Two formal convex combinations represent the same element of $\Delta_X$ if and only if they assign the same total weight to each element $x\in X$. 

\section{Defining convex spaces}
\label{def}

We first define convex spaces and convex maps before turning to a formal justification of these definitions and proving a certain uniqueness property of a convex space structure.

\begin{defn}
\label{convspcdefn}
A \ind{convex space}{convex space} is given by a set $\mathcal{C}$ together with a family of binary convex combination operations
\beq
cc_\lambda:\mathcal{C}\times\mathcal{C}\lra\mathcal{C},\quad\lambda\in[0,1]
\eeq
that satisfies\\
\begin{itemize}
\item The unit law:
\beqn
\label{unitlaw}
cc_0(x,y)=y
\eeqn
\item Idempotency:
\beqn
\label{idempotency}
cc_\lambda(x,x)=x
\eeqn
\item Parametric commutativity:
\beqn
\label{paramcomm}
cc_{\lambda}(x,y)=cc_{1-\lambda}(y,x)
\eeqn
\item Deformed parametric associativity:
\beqn
\label{defparaass}
cc_\lambda(cc_\mu(x,y),z)=cc_{\widetilde{\lambda}}(x,cc_{\widetilde{\mu}}(y,z))  
\eeqn
with\\
\beq
\widetilde{\lambda}=\lambda\mu,\qquad\widetilde{\mu}=\left\{\begin{array}{cl}\frac{\lambda\overline{\mu}}{\:\stackrel{}{\overline{\lambda\mu}}\:}&\textrm{ if }\lambda\mu\neq 1\\\textrm{arbitrary}&\textrm{ if }\lambda=\mu=1.\end{array}\right.
\eeq
\end{itemize}
\end{defn}

The most obvious example for this kind of structure is a vector space, with convex combinations defined via the vector space structure as $cc_\lambda(x,y)\equiv\lambda x+\overline{\lambda}y$. 

Definition~\ref{convspcdefn} is the picture of convex space that we shall work with. Usually, a convex space will be referred to simply by its underlying set $\mathcal{C}$, with the convex combination operations $cc_\lambda$ being implicit. Also, instead of $cc_\lambda(x,y)$, we will usually use the more suggestive notation
\beq
\lambda x+\overline{\lambda}y\equiv cc_\lambda(x,y)
\eeq
in which the laws\re{unitlaw}--\ref{defparaass} now read\\
\beqna
 0 x+\overline{0}y &=& y \\
 \lambda x + \overline{\lambda}x &=& x \\
 \lambda x + \overline{\lambda}y &=& \overline{\lambda}y+\overline{\overline{\lambda}}x \\
 \lambda\left(\mu x+\overline{\mu}y\right)+\overline{\lambda}z &=& \lambda\mu x+\overline{\lambda\mu}\left(\lambda\frac{\overline{\mu}}{\:\stackrel{}{\overline{\lambda\mu}}\:}y+\frac{\overline{\lambda}}{\:\stackrel{}{\overline{\lambda\mu}}\:}z\right)\quad (\lambda\mu\neq 1)
\eeqna

Also, we will occassionally use convex combinations
\beq
\sum_{i=1}^n\lambda_ix_i,\qquad \lambda_i\geq 0,\quad\sum_{i=1}\lambda_i=1
\eeq
of more than two elements. This are to interpreted as iterated binary convex combinations. Appropriate normalizations have to be inserted, e.g. for $n=3$,
\beq
\lambda_1x_1+\lambda_2x_2+\lambda_3x_3=\overline{\lambda}_3\left(\frac{\lambda_1}{\lambda_1+\lambda_2}x_1+\frac{\lambda_2}{\lambda_1+\lambda_2}x_2\right)+\lambda_3x_3.
\eeq
(Note that $\overline{\lambda}_3=\lambda_1+\lambda_2$.) Deformed parametric associativity\re{defparaass} then expresses the fact that this reduction to binary convex combinations does not depend on the order of bracketing.

\begin{defn}
\label{convexmap}
Given convex spaces $\mathcal{C}$ and $\mathcal{C}'$, a \ind{convex map}{convex map} from $\mathcal{C}$ to $\mathcal{C}'$ is a map $f:\mathcal{C}\ra\mathcal{C}'$ that commutes with the convex combination operations:
\beq
f(\lambda x+\overline{\lambda}y)=\lambda f(x)+\overline{\lambda}f(y).
\eeq
Convex spaces together with convex maps form the \ind{category of convex spaces}{category of convex spaces} $\mathtt{ConvSpc}$.
\end{defn}

For example, a map between vector spaces is convex if and only if it is affine. Therefore in this context, the words ``affine'' and ``convex'' will be used synonymously.

We now turn to the technical task of justifying these definitions. The goal here is to justify these definitions: why are the compatibility conditions\re{unitlaw} to\re{defparaass} sufficient to guarantee that the binary operations have all the properties we expect convex combinations to have? A less formally inclined reader may want to skip the remainder of this section.

So, what should a convex space formally be? Clearly, it has to be a set $\mathcal{C}$ together with some additional structure. This additional structure should make precise the intuition of an assignment
\beqn
\label{convcomb}
\mathfrak{m}:\Delta_\mathcal{C}\lra\mathcal{C},\qquad\sum_{i=1}^n\lambda_i\underline{x}_i\mapsto\sum_{i=1}^n\lambda_ix_i,
\eeqn
mapping a \emph{formal} convex combination $\left(\sum_{i=1}^n\lambda_i\underline{x}_i\right)\in\Delta_\mathcal{C}$ to an \emph{actual} convex combination $\left(\sum_{i=1}^n\lambda_ix_i\right)\in\mathcal{C}$, in such a way that the properties
\beqn
\label{convcombcon}
\mathfrak{m}(\underline{x})=x,\qquad\mathfrak{m}\left(\sum_{i=1}^n\lambda_i\,\underline{\mathfrak{m}\left(\sum_{j=1}^{m_i}\mu_{ij}\underline{x}_{ij}\right)}\,\right)=\mathfrak{m}\left(\sum_{i=1}^n\sum_{j=1}^{m_i}\lambda_i\mu_{ij}\underline{x}_{ij}\right)
\eeqn
hold. This intuition is straightforward to make precise using the theory of monads and their algebras\footnote{As pointed out by Leinster~\cite{Lei}, defining convex spaces in terms of an operad does not yield all properties that one desires; in particular, taking some convex combination of a point with itself would not necessarily give that point back. Therefore, defining them as algebras of a monad seems like the most canonical choice.}. The following definition is a discrete version of the Giry monad studied in categorical probability theory~\cite{Gir}.

\begin{defn}[the finitary Giry monad]
We define the simplex functor $\Delta$ to be given by
\beq
\Delta:\mathtt{Set}\ra\mathtt{Set},\quad \mathcal{C}\mapsto \Delta_\mathcal{C},\quad\left(\mathcal{C}\stackrel{f}{\ra}\mathcal{D}\right)\mapsto\left(\sum_i\lambda_i\underline{x}_i\mapsto\sum_i\lambda_i\underline{f(x_i)}\right).
\eeq
Then the \ind{finitary Giry monad}{finitary Giry monad} $\mathscr{G}_\mathrm{fin}=(\Delta,\eta,\mu)$ is defined by the unit natural transformation
\beq
\eta_\mathcal{C}:\mathcal{C}\ra \Delta_\mathcal{C},\quad x\mapsto\underline{x}
\eeq
and the multiplication transformation
\beq
\label{monadmult}
\mu_\mathcal{C}:\Delta_{\Delta_\mathcal{C}}\ra\Delta_\mathcal{C},\quad \sum_{i=1}^n\lambda_i\,\underline{\sum_{j=1}^{m_i}\mu_{ij}\underline{x}_{ij}}\mapsto\sum_{i=1}^n\lambda_i\sum_{j=1}^{m_i}\mu_{ij}\underline{x}_{ij}
\eeq
\end{defn}

An algebra of $\mathscr{G}_\mathrm{fin}$ is given by a set $\mathcal{C}$ together with a structure map $\mathfrak{m}:\Delta_\mathcal{C}\ra\mathcal{C}$, such that the diagrams
\beqn
\label{monadalgebra}
\vcenter{\xymatrix{{}\mathcal{C}\ar@{=}[rr]\ar[rd]_{\eta_\mathcal{C}} && {}\mathcal{C} && {}\Delta_{\Delta_\mathcal{C}}\ar[r]^{\Delta_{\mathfrak{m}}}\ar[d]_{\mu_\mathcal{C}} & {}\Delta_\mathcal{C}\ar[d]^{\mathfrak{m}}\\& {}\Delta_\mathcal{C}\ar[ur]_{\mathfrak{m}} &&& {}\Delta_\mathcal{C}\ar[r]^{\mathfrak{m}} & {}\mathcal{C}}}
\eeqn
commute. As can be seen directly from the definition of $\mathscr{G}_\mathrm{fin}$, these correspond exactly to the requirements\re{convcombcon}. Hence, one definitively ``correct'' definition of convex space is given by
\beq
\textrm{convex space}=\mathscr{G}_\mathrm{fin}\textrm{-algebra}.
\eeq

\begin{rem}
Since most of the applications we have in mind do not require convex combinations of infinitely many elements, it is sufficient to work with this finitary version of the Giry monad. The advantage of this is that it gives a purely algebraic description of convex spaces, thereby facilitating the reformulation~\ref{convspcdefn}. However for applications in which one needs a structure that allows to take convex combinations of infinitely many points, or more generally taking the barycenter of an arbitrary probability measure, one could define an \emph{ultraconvex space} to be an algebra of the Giry monad $\mathscr{G}$ based on the functor $\mathcal{P}:\mathtt{Meas}\ra\mathtt{Meas}$, where $\mathtt{Meas}$ is an appropriate category of measurable spaces. $\mathcal{P}$ maps each measurable space to the set of all its probability measures, together with an appropriate $\sigma$-algebra on that set. Algebras for the Giry monad over the category of polish spaces have been studied in~\cite{Dob}.
\end{rem}

We now turn to the category of stochastic matrices $\mathtt{FinStoMap}$ that was introduced in~\cite{pcsm}. We will see later that a structure\re{convcomb} satisfying\re{convcombcon} also turns $\mathcal{C}$ uniquely into a model of the Lawvere theory $\mathtt{FinStoMap}^\mathrm{op}$, and vice versa. So, we now proceed to study what it means for a functor $L:\mathtt{FinStoMap}^\mathrm{op}\lra\mathtt{Set}$ to be product-preserving. For any $\mathcal{C}\in\mathtt{Set}$, consider the functor
\beqa
\prod_\mathcal{C}:\mathtt{FinMap}^\mathrm{op} & \lra & \mathtt{Set},\qquad [n]\mapsto \mathcal{C}^{\times n}\\
\quad \left([m]\stackrel{f}{\ra}[n]\right)^\mathrm{op} & \mapsto & \left((x_1,\ldots,x_n)\mapsto(x_{f(1)},\ldots,x_{f(m)})\right).
\eeqa
Using the notation of~\cite{pcsm}, the following well-known observation arises:

\begin{prop}
\label{prodpreschar}
Consider a functor $L:\mathtt{FinStoMap}^\mathrm{op}\lra\mathtt{Set}$ with $L([n])=\mathcal{C}^{\times n}$ for all $n\in\N_0$. Then the following conditions are equivalent:
\begin{compactenum}
\item $L$ is product-preserving, i.e.
\beqn
\label{coprodtoprod}
\begin{array}{c}L\left(\left(\begin{array}{c}\mathbbm{1}_{n_1}\\0\end{array}\right)\right)=\left(\mathcal{C}^{\times n_1}\times \mathcal{C}^{\times n_2}\stackrel{p_1}{\lra}\mathcal{C}^{\times n_1}\right)\\\\
L\left(\left(\begin{array}{c}0\\\mathbbm{1}_{n_2}\end{array}\right)\right)=\left(\mathcal{C}^{\times n_1}\times \mathcal{C}^{\times n_2}\stackrel{p_2}{\lra}\mathcal{C}^{\times n_2}\right)\end{array}
\eeqn
for all $n_1,n_2\in\N_0$, where $p_1$ and $p_2$ are the product projections in $\mathtt{Set}$.
\item $L$ maps $\otimes$ to $\times$.
\item The diagram
\beqn
\label{prodpresdiagram}
\vcenter{\xymatrix{{}\mathtt{FinMap}^\mathrm{op}\:\ar@{^{(}->}[rr]\ar[rd]_{\prod_\mathcal{C}} && {}\mathtt{FinStoMap}^\mathrm{op}\ar[dl]^L\\
& {}\mathtt{Set}}}
\eeqn
commutes.
\end{compactenum}
\end{prop}

\begin{proof}
\underline{(a)$\Rightarrow$(b):} This follows from an application of $L$ to the $\mathtt{FinStoMap}$-coproduct diagram
\beq
\xymatrix{[n_1]\ar@{^{(}->}[d]\ar[rr]^{f_1} && [m_1]\ar@{_{(}->}[d]\\
[n_1+n_2]\ar[rr]^{f_1\otimes f_2} && [m_1+m_2]\\
[n_2]\ar@{_{(}->}[u]\ar[rr]^{f_2} && [m_2]\ar@{^{(}->}[u]}
\eeq
together with the product universal property in $\mathtt{Set}$.

\underline{(b)$\Rightarrow$(c):} Since $L(\partial)$ is necessarily the unique map $\mathcal{C}\ra\ast$, we know that the map
\beq
L(\partial^{\otimes k}\otimes\id_{[1]}\otimes\partial^{\otimes l}):\mathcal{C}^{\times(k+1+l)}\lra\mathcal{C}
\eeq
is the projection onto the $(k+1)$-th factor. Then for $f\in\mathtt{FinMap}([m],[n])$, the assertion follows from an application of $L$ to the equation
\beq
f(\partial^{\otimes (k-1)}\otimes\id_{[1]}\otimes\partial^{\otimes(m-k)})=\partial^{\otimes(f(k)-1)}\otimes\id_{[1]}\otimes\partial^{\otimes(n-f(k))}.
\eeq

\underline{(c)$\Rightarrow$(a):} The equations\re{coprodtoprod} are the special cases of the commutative diagram where one starts in $\mathtt{FinMap}$ with the coproduct inclusions.
\end{proof}

We now claim that the equation
\beqn
\label{lawveremonad}
L(A)(x_1,\ldots,x_n)=\left(\mathfrak{m}\left(\sum_{i=1}^nA_{i1}\underline{x}_i\right),\ldots,\mathfrak{m}\left(\sum_{i=1}^nA_{im}\underline{x}_i\right)\right)
\eeqn
uniquely determines a structure of $\mathtt{FinStoMap}^\mathrm{op}$-model $L$ on a set $\mathcal{C}$ from a $\mathscr{G}_\mathrm{fin}$-algebra structure $\mathfrak{m}:\Delta_\mathcal{C}\ra\mathcal{C}$, and vice versa. Furthermore, we claim that this correspondence is such that morphisms of $\mathscr{G}_\mathrm{fin}$-algebras coincide with morphisms of $\mathtt{FinStoMap}^\mathrm{op}$-models.

We first check that when $\mathfrak{m}$ is given, then $L$ defined by\re{lawveremonad} is a product-preserving functor. Functoriality is expressed by preservation of identities,
\beq
L(\mathbbm{1}_n)(x_1,\ldots,x_n)=\left(\mathfrak{m}(\underline{x}_1),\ldots,\mathfrak{m}(\underline{x}_n)\right)\stackrel{(\ref{monadalgebra})}{=}(x_1,\ldots,x_n),
\eeq
and contravariant preservation of matrix multiplication for $A:[m]\ra[n]$ and $B:[n]\ra[q]$. For the verification of the latter, we have to evaluate the expression
\beq
L\left(BA\right)(x_1,\ldots,x_q)=\left(\mathfrak{m}\left(\sum_{i=1}^q(BA)_{i1}\underline{x}_i\right),\ldots,\mathfrak{m}\left(\sum_{i=1}^q(BA)_{im}\underline{x}_i\right)\right).
\eeq
We do this componentwise, where $k\in[m]$ is the component index,
\beq
\left[L\left(BA\right)(x_1,\ldots,x_q)\right]_k=\mathfrak{m}\left(\sum_{i=1}^q(BA)_{ik}\underline{x}_i\right)=\mathfrak{m}\left(\sum_{i,j=1}^{q,n}B_{ij}A_{jk}\underline{x}_i\right)
\eeq
\beq
\stackrel{(\ref{monadmult})}{=}\mathfrak{m}\left(\mu_\mathcal{C}\left(\sum_{j=1}^nA_{jk}\,\underline{\sum_{i=1}^qB_{ij}\underline{x}_i}\,\right)\right)\stackrel{(\ref{monadalgebra})}{=}\mathfrak{m}\left(\sum_{j=1}^nA_{jk}\,\underline{\mathfrak{m}\left(\sum_{i=1}^qB_{ij}\underline{x}_i\right)}\,\right)
\eeq\\
\beq
\stackrel{(\ref{lawveremonad})}{=}\mathfrak{m}\left(\sum_{j=1}^nA_{jk}\,\underline{\left[L(B)(x_1,\ldots,x_q)\right]_j}\,\right)
\eeq\\
\beq
\stackrel{(\ref{lawveremonad})}{=}\left[L(A)\left(\left[L(B)(x_1,\ldots,x_q)\right]_1,\ldots,\left[L(B)(x_1,\ldots,x_q)\right]_n\right)\right]_k
\eeq\\
\beq
=\left[L(A)L(B)(x_1,\ldots,x_q)\right]_k,
\eeq
thereby showing that
\beq
L(BA)(x_1,\ldots,x_q)=L(A)L(B)\left(x_1,\ldots,x_q\right),
\eeq
which completes the verification of functoriality. Preservation of products is immediate, as the condition\re{prodpresdiagram} holds by\re{lawveremonad} and the first diagram of\re{monadalgebra}.

Now given two $\mathscr{G}_\mathrm{fin}$-algebras $\mathfrak{m}:\Delta_\mathcal{C}\ra\mathcal{C}$ and $\mathfrak{m}':\Delta_{\mathcal{C}'}\ra\mathcal{C}'$, a morphism of algebras is a map $f:\mathcal{C}\ra\mathcal{C}'$ such that the diagram
\beqn
\label{algebramorph}
\vcenter{\xymatrix{{}\Delta_\mathcal{C}\ar[r]^{\Delta_f}\ar[d]_{\mathfrak{m}} & {}\Delta_{\mathcal{C}'}\ar[d]^{\mathfrak{m}'}\\
{}\mathcal{C}\ar[r]^f & {}\mathcal{C}'}}
\eeqn
commutes. Then the induced functors $L$ and $L'$ behave with respect to $f$ in the following way:
\beq
\left[L'(A)\left(f(x_1),\ldots,f(x_n)\right)\right]_k\stackrel{(\ref{lawveremonad})}{=}\mathfrak{m}'\left(\sum_{i=1}^nA_{ik}\underline{f(x)_i}\right)=\mathfrak{m}'\left(\Delta_f\left(\sum_{i=1}A_{ik}\underline{x}_i\right)\right)
\eeq
\beq
\stackrel{(\ref{algebramorph})}{=}f\left(\mathfrak{m}\left(\sum_{i=1}^nA_{ik}\underline{x}_i\right)\right)\stackrel{(\ref{lawveremonad})}{=}f\left(\left[L(A)(x_1,\ldots,x_n)\right]\right)
\eeq
thereby showing that $L'(A)f^{\times n}=f^{\times m}L(A)$, which means that $f$ also is a morphism of $\mathtt{FinStoMap}^\mathrm{op}$-models.

Now for the other direction: given $L$, equation\re{lawveremonad} requires that we define the structure map as
\beqn
\label{monadlawvere}
\mathfrak{m}\left(\sum_{i=1}^n\lambda_i\underline{x}_i\right)\equiv L\left(\left(\begin{array}{c}\lambda_1\\\vdots\\\lambda_n\end{array}\right)\right)(x_1,\ldots,x_n)=L\big(\vec{\lambda}\big)(x_1,\ldots,x_n).
\eeqn
We need to verify the desired properties\re{monadalgebra}. The unit condition is essentially trivial,
\beq
\mathfrak{m}\left(\underline{x}\right)=L\left(\mathbbm{1}_1\right)(x)=x
\eeq
while the associativity of the action requires more work:\\
\beq
\mathfrak{m}\left(\mu_\mathcal{C}\left(\sum_{i=1}^n\lambda_i\,\underline{\sum_{j=1}^m\mu_{ji}\underline{x}_j}\,\right)\right)\stackrel{(\ref{monadmult})}{=}\mathfrak{m}\left(\sum_{i=1}^n\lambda_i\sum_{j=1}^m\mu_{ji}\underline{x}_j\right)
\eeq\\
\beq
\stackrel{(\ref{monadlawvere})}{=}L\left(\left(\begin{array}{c}\sum_{i=1}^n\lambda_i\mu_{1i}\\\vdots\\\sum_{i=1}^n\lambda_i\mu_{mi}\end{array}\right)\right)(x_1,\ldots,x_m)
\eeq\\
\beq
=L\big(\underline{\mu}\vec{\lambda}\big)(x_1,\ldots,x_m)=L\big(\vec{\lambda}\big)L(\underline{\mu})(x_1,\ldots,x_m)
\eeq\\
where the matrix $\underline{\mu}=\left(\mu_{ji}\right)_{j,i}$ has columns $\vec{\mu}_1,\ldots,\vec{\mu}_n$, and after possibly adding dummy terms, we were able to assume that under the large underscore, neither the number of terms $m$ nor the $x_j$ depend on $i$. Since $L$ maps coproducts to products, and the columns of the matrix $\underline{\mu}$ are exactly its coproduct components, we can continue the calculation with
\beq
=L\big(\vec{\lambda}\big)\left(L\big(\vec{\mu}_1\big)(x_1,\ldots,x_m),\ldots,L\big(\vec{\mu}_n\big)(x_1,\ldots,x_m)\right)
\eeq
\beq
\stackrel{(\ref{monadlawvere})}{=}\mathfrak{m}\left(\sum_{i=1}^n\lambda_i\,\underline{L\big(\vec{\mu}_i\big)(x_1,\ldots,x_m)}\,\right)\stackrel{(\ref{monadlawvere})}{=}\mathfrak{m}\left(\sum_{i=1}\lambda_i\,\underline{\mathfrak{m}\left(\sum_{j=1}^m\mu_{ji}\underline{x}_j\right)}\,\right)
\eeq\\
which shows that also the second diagram of\re{monadalgebra} commutes.

What still remains to check is that morphisms of $\mathtt{FinStoMap}^\mathrm{op}$-models also are morphisms of the induced $\mathscr{G}_\mathrm{fin}$-algebras. This follows from essentially the same calculation as above:\\
\beq
\mathfrak{m}'\left(\Delta_f\left(\sum_{i=1}A_{ik}\underline{x}_i\right)\right)=\mathfrak{m}'\left(\sum_{i=1}^nA_{ik}\underline{f(x)_i}\right)=\left[L'(A)\left(f(x_1),\ldots,f(x_n)\right)\right]_k
\eeq\\
\beq
=f\left(\left[L(A)(x_1,\ldots,x_n)\right]\right)=f\left(\mathfrak{m}\left(\sum_{i=1}^nA_{ik}\underline{x}_i\right)\right).
\eeq

Finally, as the observation concluding these considerations, it follows from the uniqueness statement of the correspondence $\mathfrak{m}\leftrightsquigarrow L$ that the construction of $L$ from $\mathfrak{m}$ is inverse to the construction of $\mathfrak{m}$ from $L$.

\begin{rem}
This correspondence between algebras of a monad and models of a Lawvere theory is a particular instance of a well-known general correspondence between finitary monads and Lawvere theories~\cite{HP}. (A monad is called \emph{finitary} if the endofunctor preserves filtered colimits.)
\end{rem}

Hence, we now have two definitively correct possible definitions of convex space: a $\mathcal{G}_\mathrm{fin}$-algebra, or a model of $\mathtt{FinStoMap}^\mathrm{op}$. We can now apply theorem~\cite[3.14]{pcsm} to show that the compatibility requirements of definition~\ref{convspcdefn} do indeed give all the relations~\ref{convcombcon} that we expect convex combinations to have.

\begin{prop}
\label{makeconcrete}
Given a set $\mathcal{C}$ together with a structure of $\mathtt{FinStoMap}^\mathrm{op}$-model in terms of a product-preserving functor $L:\mathtt{FinStoMap}^\mathrm{op}\lra\mathtt{Set}$, the operations
\beqn
\label{ccfromc}
cc_\lambda\equiv L(c_\lambda)
\eeqn
define the structure of a convex space on $\mathcal{C}$. Conversely given $cc_\lambda$, there is a unique $L$ such that\re{ccfromc} holds.
\end{prop}

\begin{proof}
This is the main application of theorem~\cite[3.14]{pcsm}. First note that due to proposition~\ref{prodpreschar}, any product-preserving $L$ satisfies
\beqa
L(\partial):&\mathcal{C}\ra\ast,\quad& x\mapsto\ast\\
L(e):&\mathcal{C}\ra\mathcal{C}\times\mathcal{C},\quad& x\mapsto(x,x)\\
L(s):&\mathcal{C}\times\mathcal{C}\ra\mathcal{C}\times\mathcal{C},\quad& (x,y)\mapsto(y,x)
\eeqa
Hence, $L$ is automatically compatible with the relations~\cite[(2)-(7), (11), (12)]{pcsm}.

However, $L$ also needs to preserve the other relations of $\mathtt{FinStoMap}'$. In exactly this order, preservation of each of the relations~\cite[(8), (9), (10) and (13)]{pcsm} is equivalent to one of the requirements\re{unitlaw} to\re{defparaass}.
\end{proof}

We now turn to proving that the category $\mathtt{ConvSpc}$ enjoys a certain rigidity property expressed by theorem~\ref{convspc}.

For the following lemma, consider the family of maps on the unit interval $[0,1]$ that is given by
\beq
f_{y_0,y_1}:[0,1]\lra[0,1],\qquad x\mapsto\overline{x}y_0+xy_1,\quad y_0,y_1\in[0,1].
\eeq

\begin{lem}
\label{unitinterval}
\begin{compactenum}
\item The unit interval $[0,1]$ has a unique structure of convex space in which all of the $f_{y_0,y_1}$ are convex maps.
\item For every convex space $\mathcal{C}$ and every pair of points $x,y\in\mathcal{C}$, there is a unique convex map $g_{x,y}:[0,1]\ra\mathcal{C}$ with $g(0)=x$ and $g(1)=y$.
\end{compactenum}
\end{lem}

\begin{proof}
(a) In order to distinguish elements of the convex space $[0,1]$ from coefficients in $[0,1]$, we distinguish the fomer by means of the underline notation $\underline{\cdot}$.

We first show that the convex combination $\frac{1}{2}\underline{0}+\frac{1}{2}\underline{1}$ is necessarily equal to $\underline{1/2}$. To this end, consider the flip map $f_{1,0}$:
\beq
f_{1,0}\left(\frac{1}{2}\underline{0}+\frac{1}{2}\underline{1}\right)=\frac{1}{2}f_{1,0}\left(\underline{0}\right)+\frac{1}{2}f_{1,0}\left(\underline{1}\right)=\frac{1}{2}\underline{1}+\frac{1}{2}\underline{0}=\frac{1}{2}\underline{0}+\frac{1}{2}\underline{1}
\eeq
Hence, the assertion follows from the fact that $\underline{1/2}$ is the unique fixed point of $f_{1,0}$.

But then also for any pair $x,y\in[0,1]$, we have that
\beq
\frac{1}{2}\underline{x}+\frac{1}{2}\underline{y}=\frac{1}{2}f_{x,y}\left(\underline{0}\right)+\frac{1}{2}f_{x,y}\left(\underline{1}\right)=f_{x,y}\left(\frac{1}{2}\underline{0}+\frac{1}{2}\underline{1}\right)=f_{x,y}\left(\underline{\frac{1}{2}}\right)=\underline{\frac{1}{2}x+\frac{1}{2}y}
\eeq

Next, we claim that when $x<y$, $p,q\in\N_0$ and $\lambda\in(0,1)$ with $q2^{-p}\leq\lambda\leq(q+1)2^{-p}$, then
\beqn
\label{ininterval}
\left(\overline{\lambda}\underline{x}+\lambda\underline{y}\right)\in\left[\overline{q2^{-p}}x+q2^{-p}y,\,\overline{(q+1)2^{-p}}x+(q+1)2^{-p}y\right]
\eeqn
We prove this by induction on $p$. For $p=0$, this is given by
\beq
\overline{\lambda}\underline{x}+\lambda\underline{y}=f_{x,y}\left(\overline{\lambda}\underline{0}+\lambda\underline{1}\right)\in\mathrm{im}\left(f_{x,y}\right)=[x,y].
\eeq
For $p\geq 1$, consider the case $\lambda\geq 1/2$ first, which is equivalent to $q\geq 2^{p-1}$. Then 
\beq
\left(q-2^{p-1}\right)2^{-(p-1)}\leq 2\lambda-1\leq\left(q+1-2^{p-1}\right)2^{-(p-1)}
\eeq
so that
\beq
\overline{\lambda}\underline{x}+\lambda\underline{y}=2\overline{\lambda}\left(\frac{1}{2}\underline{x}+\frac{1}{2}\underline{y}\right)+(2\lambda-1)\underline{y}=2\overline{\lambda}\:\underline{\left(\frac{1}{2}x+\frac{1}{2}y\right)}+(2\lambda-1)\underline{y}
\eeq
which, by the induction assumption, is bigger than or equal to
\beq
\overline{(q-2^{p-1})2^{-(p-1)}}\left(\frac{1}{2}x+\frac{1}{2}y\right)+(q-2^{p-1})2^{-(p-1)}y=\overline{q2^{-p}}x+q2^{-p}y,
\eeq
as was to be shown. The upper bound works in exactly the same way. The case $\lambda\leq 1/2$ can either be treated in a similar way, or can be reduced to the case $\lambda\geq 1/2$ by an application of the flip map $f_{1,0}$.

But then by the principle of nested intervals, equation\re{ininterval} shows that $\overline{\lambda}\underline{x}+\lambda\underline{y}=\underline{\overline{\lambda}x+\lambda y}$, which concludes the proof.

(b) For $\lambda\in[0,1]$, the requirements imply that we need to set
\beq
g(\lambda)\equiv\overline{\lambda}x+\lambda y.
\eeq
We now verify that this is indeed a convex map. With $\mu,\lambda_1,\lambda_2\in[0,1]$, we have
\beqn
\label{twopointconvex}
g\left(\mu\lambda_1+\overline{\mu}\lambda_2\right)=\overline{\left(\mu\lambda_1+\overline{\mu}\lambda_2\right)}\,x+\left(\mu\lambda_1+\overline{\mu}\lambda_2\right)y.
\eeqn
We proceed by evaluating the first coefficient further,
\beq
\overline{\left(\mu\lambda_1+\overline{\mu}\lambda_2\right)}=\overline{\mu\lambda_1}+\overline{\overline{\mu}\lambda_2}-1
\eeq
\beq
=\overline{\mu}+\overline{\lambda}_1-\overline{\mu}\overline{\lambda}_1+\mu+\overline{\lambda}_2-\mu\overline{\lambda}_2-1=\mu\overline{\lambda}_1+\overline{\mu}\overline{\lambda}_2
\eeq
proving that\re{twopointconvex} yields
\beq
g\left(\mu\lambda_1+\overline{\mu}\lambda_2\right)=\mu\left(\overline{\lambda}_1x+\lambda_1y\right)+\overline{\mu}\left(\overline{\lambda}_2x+\lambda_2y\right)=\mu g(\lambda_1)+\overline{\mu}g(\lambda_2),
\eeq
as was to be shown.
\end{proof}

\begin{thm}
\label{convspc}
The identity functor is the only endofunctor of $\mathtt{ConvSpc}$ that makes the diagram
\beq
\xymatrix{{}\mathtt{ConvSpc}\ar[rr]\ar[rd] && {}\:\mathtt{ConvSpc}\ar[ld]\\
& {}\mathtt{Set}}
\eeq
commute.
\end{thm}

\begin{proof}
Let $E:\mathtt{ConvSpc}\ra\mathtt{ConvSpc}$ be such an endofunctor. Commutativity of the diagram means that for any $\mathcal{C},\mathcal{C}'\in\mathtt{ConvSpc}$, $E(\mathcal{C})$ and $E(\mathcal{C}')$ are convex spaces with the same underlying sets as $\mathcal{C}$ and $\mathcal{C}'$, respectively, such that
\beqn
\label{convmapincl}
\mathtt{ConvSpc}\left(\mathcal{C},\mathcal{C}'\right)\subseteq\mathtt{ConvSpc}\left(E(\mathcal{C}),E(\mathcal{C}')\right).
\eeqn
Now consider $\mathcal{C}=\mathcal{C}'=[0,1]$. Then it follows from lemma~\ref{unitinterval}(a) that $E([0,1])=[0,1]$ with the standard structure of convex space.

Now consider $\mathcal{C}=[0,1]$ and $\mathcal{C}'$ arbitrary. Then by lemma~\ref{unitinterval}(b), we know that for any $x,y\in\mathcal{C}'$,
\beq
\overline{\lambda}x+\lambda y=g_{x,y}(\lambda).
\eeq
Therefore, the structure of convex space on $E(\mathcal{C}')$ is uniquely determined by\re{convmapincl}, showing that $E(\mathcal{C}')=\mathcal{C}'$.
\end{proof}

\begin{rem}
Theorem~\ref{convspc} displays a rigidity of $\mathtt{ConvSpc}$ that is far from valid for other categories of algebraic structures. For example for the category of groups $\mathtt{Grp}$, there is a non-trivial automorphism $\cdot^\mathrm{op}:\mathtt{Grp}\lra\mathtt{Grp}$, given by mapping each group to its opposite group, such that the diagram
\beq
\xymatrix{{}\mathtt{Grp}\ar[rr]^{\cdot^\mathrm{op}}\ar[rd] && {}\:\mathtt{Grp}\ar[ld]\\
& {}\mathtt{Set}}
\eeq
commutes. Hence, the direct analogue of theorem~\ref{convspc} for groups is false.
\end{rem}

\section{Convex spaces of geometric type}
\label{examplesgeom}

The first main class of examples of convex spaces are the convex subsets of vector spaces, which will be discussed now. We will refer to those convex spaces that can be embedded into a vector space as \ind{convex spaces of geometric type}{convex space of geometric type}. These are the convex spaces studied in convex geometry. We are aware that many relevant properties of a convex set do depend on an explicit embedding into a vector space: for example, the volume or the number of points with integer coordinates are properties that are not invariant under affine transformations and therefore are not invariants of the convex space structure alone. Nevertheless, we hope that the theory of convex spaces~\cite{propclass} might be able to shed new light on some aspects of convex geometry in general and some of the following examples in particular.

We will see in the upcoming two sections that there are also interesting examples of convex spaces that are not of geometric type.

\begin{thm}[convex spaces of geometric type]
\label{convexsubset}
Given a real vector space $V$ and a convex subset $\mathcal{C}\subseteq V$, the vector space structure of $V$ turns $\mathcal{C}$ into a convex space.
\end{thm}

\begin{proof}
This is clear by defining the convex combination operations $cc_\lambda$ via the vector space structure in the obvious way as
\beq
cc_\lambda(x,y)\equiv\lambda x+\overline{\lambda}y
\eeq
since then the equations\re{unitlaw}--(\ref{defparaass}) follow easily from the vector space axioms.
\end{proof}

The map which turns every such convex set into a convex space is functorial in the following sense: consider the category of convex sets, where objects are pairs $(V,\mathcal{C})$ with $V$ a real vector space and $\mathcal{C}\subseteq V$ a convex subset, and the morphisms $(V,\mathcal{C})\ra(V',\mathcal{C}')$ are the affine maps $f:V\ra V'$ with $f(\mathcal{C})\subseteq\mathcal{C}'$. Then each morphism $f$ restricts to a convex map between convex spaces $f_{|\mathcal{C}}:\mathcal{C}\ra\mathcal{C}'$. This construction is clearly functorial.

All examples following now are convex spaces of geometric type. In each case, we also describe how the convex space arises as a convex subset of a vector space.

\begin{ex}[free convex spaces]
Given a set $X$, the simplex $\Delta_X$ is a convex subset of the vector space $\R^X$. Alternatively, we can regard $\Delta_X$ as the set of formal convex combinations of elements of $X$. In this interpretation, $\Delta_X$ is the ``free'' convex space generated by $X$ in the sense of a functor $\mathtt{Set}\ra\mathtt{ConvSpc}$ left adjoint to the forgetful functor $\mathtt{ConvSpc}\ra\mathtt{Set}$. This property is clear from the monadic definition of convex spaces, where $\Delta_\cdot$ figures as the underlying functor of the monad $\mathscr{G}_\mathrm{fin}$. As a third point of view, $\Delta_X$ can also be regarded as the set of finitely supported probability measures on $X$.
\end{ex}

\begin{ex}[probability measures]
As a variant of the previous example, we may consider a set $X$ together with any $\sigma$-algebra $\Omega\subseteq 2^X$, turning $(X,\Omega)$ into a measurable space. Then the set of probability measures on $(X,\Omega)$ is a convex subset of the vector space $\R^\Omega$. We denote this convex space by $\Delta_{(X,\Omega)}$.
\end{ex}

\begin{ex}[invariant measures]
Let $(X,\Omega)$ be a measurable space together with an action of a group $G$ or monoid $G$ given by a homomorphism $G\ra\mathrm{End}(X)$. For example when $G=(\R,+)$, this action turns $X$ into a dynamical system. Then the set of \emph{invariant measures}, which are those probability measures that are preserved by the action of $G$, form a convex subspace of $\Delta_{(X,\Omega)}$. Of particular importance are the \emph{ergodic measures} as those that cannot be written as a non-trivial convex combination of other invariant measures.
\end{ex}

\begin{ex}[conditional probability distributions / classical communication channels]
Given measurable spaces $(X,\Omega_X)$ and $(Y,\Omega_Y)$, a \emph{conditional probability distribution} on $Y$ dependent on $X$ is defined to be a convex map $\Delta_{(X,\Omega_X)}\ra\Delta_{(Y,\Omega_Y)}$. Such a map describes a classical communication channel, where an input $x\in X$ is represented by the Dirac measure on $x$ and gets mapped to a probability distribution of noise-affected possible outputs $y\in Y$. The set of all such maps is a convex space under pointwise convex combinations. 
\end{ex}

\begin{ex}[states on $C^*$-algebras]
Given a $C^*$-algebra $A$, a \emph{state on} $A$ is a positive linear functional $\phi:A\ra\C$ of unit norm. The states on $A$ form a convex subset of the vector space $\C^A$. In the case $A=\mathcal{B}(\mathcal{H})$, this convex space is isomorphic to the convex set of unit trace positive trace-class operators on $\mathcal{H}$, the so-called \ind{density matrices}{density matrix}. Upon setting $\mathcal{H}_n\equiv\C^n$ for $n\in\N$ and $H_n\equiv\ell^2(\N)$ for $n=\infty$, the set of density matrices is given by
\beq
\mathcal{Q}_n\equiv\left\{\rho\in\mathcal{B}(\mathcal{H}_n)\:|\:\rho\geq 0,\:\mathrm{tr}(\rho)=1\right\}.
\eeq
This family of convex spaces is widely studied in quantum information theory. As a first example of how much information the convex space structure on $\mathcal{Q}_n$ contains, we show that one can use it to recover the scalar product of $\mathcal{H}_n$, at least up to a phase factor. This is achieved by the formula, depending on unit vectors $|\psi_1\rangle$ and $|\psi_2\rangle$,
\beqn
\label{scalarprod}
\sqrt{1-|\langle\psi_1|\psi_2\rangle|^2}=\max_{f:\mathcal{Q}_n\ra[0,1]\textrm{ convex}}\Big|f\left(|\psi_1\rangle\langle\psi_1|\right)-f\left(|\psi_2\rangle\langle\psi_2|\right)\Big|.
\eeqn
In order to prove the correctness of this equation, we consider the case $n=2$ first. Then $|\psi_1\rangle$ and $\psi_2\rangle$ can be identified with points on the Bloch sphere. The angle between these points, as seen from the center of the sphere, is given by
\beq
\cos\alpha=|\langle\psi_1|\psi_2\rangle|^2,\quad\alpha\in[0,\pi]
\eeq
since the map $\rho\mapsto\mathrm{tr}(\rho|\psi_1\rangle\langle\psi_1|)$ is convex and can therefore be identified with a cartesian coordinate for the sphere. This situation is illustrated in figure~\ref{blochsphere}. 

\begin{figure}
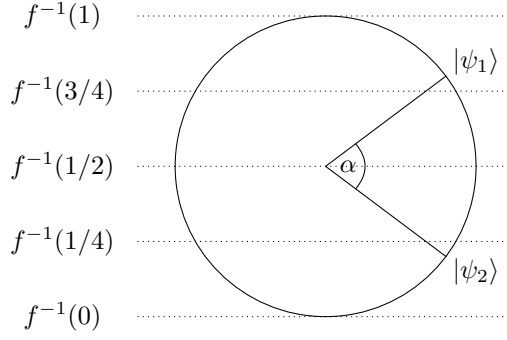

\begin{center}
\beq\xy
(20,0);(0,0),{\ellipse<>{-}};
(16,12) **@{-} +(4,2)*={|\psi_1\rangle},
(16,-12) **@{-} +(4,-2)*={|\psi_2\rangle},
(4,-3);(4,3) **\crv{(6.5,0)},
(3,0)*={\alpha},
(25,20);(-25,20) **@{.} -(10,0)*={f^{-1}(1)},
(25,10);(-25,10) **@{.} -(10,0)*={f^{-1}(3/4)},
(25,0);(-25,0) **@{.} -(10,0)*={f^{-1}(1/2)},
(25,-10);(-25,-10) **@{.} -(10,0)*={f^{-1}(1/4)},
(25,-20);(-25,-20) **@{.} -(10,0)*={f^{-1}(0)},
\endxy\eeq
\end{center}
\caption{A two-dimensional section of the Bloch ball containing the states $|\psi_1\rangle$ and $|\psi_2\rangle$. The level sets of the optimal functional $f$ are shown with pointed lines.}
\label{blochsphere}
\end{figure}

Now when $f$ is a $[0,1]$-valued convex functional on the Bloch ball, the value $|f\left(|\psi_1\rangle\langle\psi_1|\right)-f\left(|\psi_2\rangle\langle\psi_2|\right)|$ is maximal at most when $f$ attains both $0$ and $1$. Then we call $f^{-1}(1)$ the ``north pole'' and $f^{-1}(0)$ the ``south pole''; these points are clearly unique and diametrically opposite. Also it is clear that an optimal $f$ will be such that $|\psi_1\rangle$ and $|\psi_2\rangle$ are aligned symmetrically with respect to the equator. Then,
\beq
f(|\psi_1\rangle\langle\psi_1|)=\frac{1}{2}+\frac{1}{2}\sin\left(\frac{\alpha}{2}\right),\quad f(|\psi_2\rangle\langle\psi_2|)=\frac{1}{2}-\frac{1}{2}\sin\left(\frac{\alpha}{2}\right)
\eeq
so that
\beq
\left|f\left(|\psi_1\rangle\langle\psi_1|\right)-f\left(|\psi_2\rangle\langle\psi_2|\right)\right|=\sin\left(\frac{\alpha}{2}\right)=\sqrt{1-|\langle\psi_1|\psi_2\rangle|^2}\,,
\eeq
as was to be shown.

For general $n$, consider the Hilbert space spanned by $|\psi_1\rangle$ and $|\psi_2\rangle$. When $|\psi_1\rangle$ and $|\psi_2\rangle$ are linearly dependent,\re{scalarprod} holds trivially, hence we may assume the span to be two-dimensional. This yields an embedding $\mathcal{Q}_2\hookrightarrow\mathcal{Q}_n$. In this way, every convex functional $\mathcal{Q}_n\ra[0,1]$ can be restricted to $\mathcal{Q}_2\ra[0,1]$, and then the ``$\geq$'' part of\re{scalarprod} follows from the previous considerations. On the other hand, the $f$ constructed in the two-dimensional case is of the form $\rho\mapsto\langle\psi|\rho|\psi\rangle$, where $|\psi\rangle$ is an appropriate linear combination of $|\psi_1\rangle$ and $|\psi_2\rangle$. Therefore, this optimal $f$ can actually be extended to all of $\mathcal{Q}_n$, so that this ``$\geq$'' bound is in fact tight.
\end{ex}

\begin{ex}[KMS states]
A KMS state is a certain kind of state on a $C^*$-algebra relevant for equilibrium thermodynamics.

In statistical physics, thermal equilibrium of a system with its environment is described by an equilibrium state depending on the temperature. This state is usually given by the canonical ensemble's density matrix $\rho=\mathcal{Z}(\beta)^{-1}e^{-\beta H}$, where $\beta=1/kT$ is the inverse temperature of the system, $H$ stands for its Hamiltonian, and $\mathcal{Z}(\beta)=\mathrm{tr}(e^{-\beta H})$ denotes the partition function. However in some cases, the trace in the definition of $\mathcal{Z}(\beta)$ need not converge, such that the canonical ensemble does not exist. For example in the context of spontaneous symmetry breaking, there is clearly no unique equilibrium state. In these situations, equilibrium thermodynamics has to be phrased in terms of KMS states.

We now describe the notion of KMS state in detail. On the quantum level, a system is described by its $C^*$-algebra of observables $A$ and a one-parameter group of automorphisms $\alpha_t:A\ra A$; typically, this group is given by the Heisenberg picture time evolution $\alpha_t(a)=e^{iHt}ae^{-iHt}$. Then by definition, a state $\varphi:A\ra\C$ is a \emph{Kubo-Martin-Schwinger (KMS) state}~\cite[p. 178]{KM} for inverse temperature $\beta$ if and only if for all $a,b\in A$, there is a continuous function $F_{a,b}(z)$ defined on the strip $0\leq\mathrm{Im}(z)\leq\beta$, and holomorphic on the interior of the strip, such that
\beqn
\label{KMS}
F_{a,b}(t)=\varphi(a\alpha_t(b)),\qquad F_{a,b}(t+i\beta)=\varphi(\alpha_t(b)a).
\eeqn
It is then clear that the KMS states for fixed $\beta$ form a convex subset of the convex space of all states on $A$. As a plausibility check, one may observe that the canonical ensemble $\varphi(a)=\mathcal{Z}(\beta)^{-1}\mathrm{tr}(e^{-\beta H}a)$ is a KMS state whenever the partition function $\mathcal{Z}(\beta)=\mathrm{tr}(e^{-\beta H})$ converges.
\end{ex}

\begin{ex}[unit balls]
Let $\left(E,||\cdot||\right)$ be a normed space. Then the unit ball
\beq
B_1\equiv\{x\in E\:|\:||x||\leq 1\}
\eeq
is a convex space in $E$. Conversely, the convex space $B_1$ determines the norm via
\beq
||x||=\frac{1}{\sup\{r\in\R_{>0}\:|\:rx\in B_1\}}.
\eeq
The same applies to seminorms.
\end{ex}

\begin{ex}[torus actions on symplectic manifolds]
\label{torusaction}
This is material taken from the book~\cite{Aud}.

Let $(M,\omega)$ be a compact connected symplectic manifold together with a collection of Hamiltonian functions $H_1,\ldots,H_n$ such that the $H_i$ pairwise Poisson commute and generate (almost) periodic flows. Then the image of the map
\beq
f:M\ra\R^n,\quad x\mapsto\left(H_1(x),\ldots,H_n(x)\right)
\eeq
is convex.

The proof of this result follows from proposition~\ref{showgenconv} together with the statement that all the level sets $f^{-1}(t),\:t\in\R^n$, are empty or connected. The latter is a deep theorem the proof of which heavily relies on Morse theory.
\end{ex}

\begin{prop}
\label{showgenconv}
Let $X$ be a topological space and $\mathcal{F}$ a collection of functions $f:X\ra\R^{n_f}$ such that
\begin{compactitem}
\item $\mathcal{F}$ is closed under composition with linear projection maps $\:\R^{n_1}\twoheadrightarrow\R^{n_2}$,
\item all level sets $f^{-1}(t)$, $f\in\mathcal{F}$, $t\in\R^{n_f}$, are empty or connected.
\end{compactitem}
Then $\:\mathrm{im}(f)\subseteq\R^{n_f}$ is convex for every $f\in\mathcal{F}$.
\end{prop}

\begin{proof}
(see also~\cite[p. 114]{Aud}.) We need to show that the intersection of $\mathrm{im}(f)$ with every affine line in $\R^{n_f}$ is connected. To this end, choose such an affine line and some linear projection $\pi:\R^{n_f}\twoheadrightarrow\R^{n_f-1}$ that maps this affine line to a point. The inverse image of this point under $\pi$ is just the given affine line. Then by assumption, the preimage of this affine line in $X$ has to be connected, therefore showing that the intersection of $\mathrm{im}(f)$ with this affine line also is connected.
\end{proof}

The statement of the next example can be proven by applying a certain refinement of example~\ref{torusaction}. We refer to~\cite[IV.4.11]{Aud} for more details.

\begin{ex}[the Schur-Horn theorem]
Consider an $n$-tuple of not necessarily distinct numbers $\lambda_1,\ldots,\lambda_n\in\R$. Then there is a hermitian $n\times n$-matrix $A$ with $\mathrm{diag}(A)=\left(a_1,\ldots,a_n\right)\in\R^n$ and eigenvalues $\lambda_1,\ldots,\lambda_n$ if and only if
\beq
\left(a_1,\ldots,a_n\right)\in\mathrm{conv}\left(\left\{\,(\lambda_{\sigma(1)},\ldots,\lambda_{\sigma(n)}),\:\sigma\in S_n\right\}\right)
\eeq
where $\mathrm{conv}(\cdot)$ stands for the convex hull in $\R^n$ of its argument and $S_n$ is the group of permutations of $[n]$.
\end{ex}

\begin{ex}[metrics]
These are actually two related examples. For the first, let $X$ be a set. A metric on $X$ is a function $d:X\times X\lra\R_{\geq 0}$ satisfying definiteness, symmetry, and the triangle inequality. A convex combination of two metrics is again a metric. Therefore, the set of metrics is a convex space of geometric type lying in the vector space $\R^{X\times X}$.

For the second example, consider a manifold $M$ and the set of Riemannian metrics on $M$. A Riemannian metric is a positive definite symmetric tensor of rank $(0,2)$ on $M$. Therefore, the set of Riemannian metrics is a convex space of geometric type lying the vector space $\mathcal{T}^0_2(M)$ of all rank $(0,2)$ tensors on $M$.
\end{ex}

\begin{ex}[non-example: points on a Riemannian manifold]
Take $\mathcal{C}$ to be a subset of a Riemannian manifold, such that each pair of points $x,y\in\mathcal{C}$ can be joined by a unique geodesic $[x,y]\subseteq\mathcal{C}$. Upon fixing the affine parameter $\lambda$ of the geodesic $[a,b]$ such that $\lambda=0$ at $y$ and $\lambda=1$ at $x$, one might be tempted to define the convex combination $\lambda x+\overline{\lambda}y$ as the point on $[a,b]$ corresponding to the affine parameter $\lambda$. Then this satisfies the unit law, idempotency and parametric commutativity. Now assume that deformed parametric associativity also holds, thereby turning $\mathcal{C}$ into a convex space. Then any triple of points $x,y,z\in\mathcal{C}$ defines a convex map $\Delta_3\ra\mathcal{C}$ that maps straight lines to geodesics. But then by virtue of the geodesic deviation equation, the manifold is flat along the triangle spanned by $x$, $y$ and $z$. Since this triple was arbitrary, the manifold is flat on all of $\mathcal{C}$. Conversely if the manifold is flat on $\mathcal{C}$, we are exactly in the situation of theorem~\ref{convexsubset}.
\end{ex}

\begin{ex}[color perception and chromaticity]
The physical color of light is given by its spectral density $I(\lambda)$, where $I(\lambda)d\lambda$ is the intensity of light in the wavelength interval $[\lambda,\lambda+d\lambda]$. Hence a priori, there are infinitely many physical degrees of freedom in the spectrum. However since the human eye only has three different kinds of receptors, our perception projects this two a three-dimensional space, which we perceive as three different kinds of visual colors.

\begin{figure}
\begin{center}
\includegraphics[height=6cm]{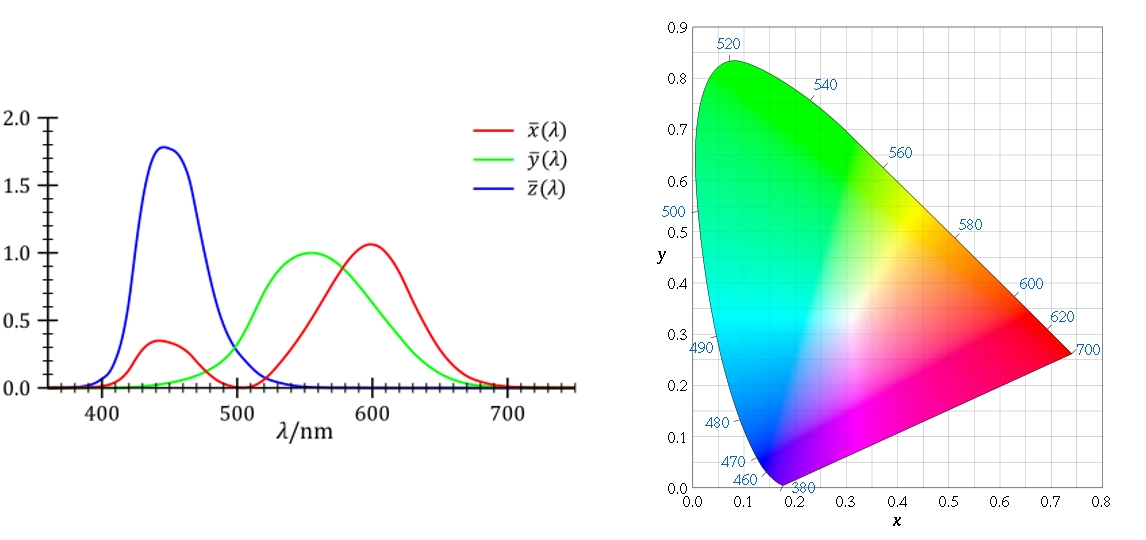}
\end{center}
\caption{The CIE 1931 color matching functions and the resulting chromaticity diagram\protect\footnotemark. The curved part of the boundary is formed by the monochromatic colors of the specified wavelengths.}
\label{colorperception}
\end{figure}

\footnotetext{Both images were copied from \url{http://en.wikipedia.org/wiki/CIE_1931_color_space} using the GNU FDL.}

More formally, a physical color is defined by a finite measure $d\mu$ on the space of wavelengths $[0,\infty)$. The corresponding visual color is obtained by integrating $d\mu$ with respect to three non-negative \emph{color matching functions}\footnote{Note that for technical reasons, these do actually not coincide with the response functions of the eye's receptors.} $\overline{x}(\lambda)$, $\overline{y}(\lambda)$, $\overline{z}(\lambda)$:
\beqa
&X=\int \overline{x}(\lambda)d\mu\\
&Y=\int \overline{y}(\lambda)d\mu\\
&Z=\int \overline{z}(\lambda)d\mu.
\eeqa
Hence we get a convex map from the convex space of all finite measures on $[0,\infty)$ to the convex space $\R_{\geq 0}^3$, such that scaling the measure by a non-negative constant also scales all $(X,Y,Z)$ by that constant. The chromaticity diagram in figure~\ref{colorperception} depicts the image of this convex map in a two-dimensional cross-section which corresponds to restricting to colors of specified brightness. Since the image of any convex map is convex, so is the color region of the chromaticity diagram. Morally speaking, we can think of any physical color $d\mu$ as a free convex combination of monochromatic colors, i.e. Dirac measures on $[0,\infty)$. Then every visual color in the chromaticity diagram is a convex combination of monochromatic colors.
\end{ex}

Convex sets also feature prominently in many kinds of optimization problems. We start with a particular example of a linear programming problem.

\begin{ex}[static friction for rigid bodies]
Consider a long and thin rod with quadratic cross-section lying on a flat surface. Then upon application of a small force along the side of the rod, the static friction between the rod and the surface keeps the rod from sliding. The question is: under the assumption that the force applies on the side of the rod towards its end, how big can that force be without the rod starting to slide? The situation is illustrated in figure~\ref{staticfriction}.

We assume all physical parameters (mass and length of the rod, coefficient of friction, \ldots) to be known and set them to unity without loss of generality. Then as shown in the figure, the friction forces along the rod are described in terms of a linear density $f(x)$ with the constraint that there is a maxmial amount of friction for each length element, so that $|f(x)|\leq 1$. Now upon application of a small enough force $\vec{F}$, the friction will adjust in such a way that the force is balanced, i.e. $\vec{F}+\vec{e}_y\int_0^1f(x)dx=0$, and torque is balanced, i.e. $\int_0^1xf(x)dx=0$. Hence the maximal force that can be applied is given by the solution of the linear program
\beqa
&-1\leq f(x)\leq +1\\
&\int_0^1xf(x)dx=0\\
&\max\left(\int_0^1f(x)dx\right)
\eeqa
As always in linear programming, the set of admissible solutions $f(x)$ is determined by a set of linear equalities and inequalities, and therefore is convex. We can solve this problem by introducing a Lagrange multiplier $\mu$ for the equality constraint, and solving the optimization problem
\beqa
&-1\leq f(x)\leq +1\\
&\max\left(\int_0^1f(x)dx+\mu\int_0^1xf(x)dx\right)=\max\left(\int_0^1\left(\mu x+1\right)f(x)dx\right)
\eeqa
It is clear this problem has a unique optimal solution given by
\beq
f^*_\lambda(x)=\left\{\begin{array}{cc}+1&\textrm{for }\:\mu x+1>0\\-1&\textrm{for }\:\mu x-1<0\end{array}\right..
\eeq
Then the torque constraint $\int_0^1xf(x)dx=0$ holds if and only if $\mu=-\sqrt{2}$, so that the optimal configuration is given by
\beq
f^*(x)=\left\{\begin{array}{cc}+1&\textrm{for }\: x<1/\sqrt{2}\\-1&\textrm{for }\: x>1\sqrt{2}\end{array}\right..
\eeq
With this result, we determine the absolute value of the maximal force to be 
\beq
F=\int_0^1f^*(x)dx=\sqrt{2}-1.
\eeq

We expect that these considerations can be generalized to arbitrary rigid bodies in $\R^n$. To this end, $f$ will have to be replaced by a vector-valued function $\vec{f}(x)$ restricted such as $|\vec{f}(x)|\leq\rho(x)$, where $\rho$ is the rigid body's density distribution, while there will be one linear constraint for each component of the total torque. Then the set of admissible $\vec{f}(x)$ is a convex space that comes with a convex map to the vector space of all potential forces acting on a certain point of the rigid body. The forces that can be applied at that point without the body starting to slide are exactly given by the image of this convex map.
\end{ex}

\begin{figure}
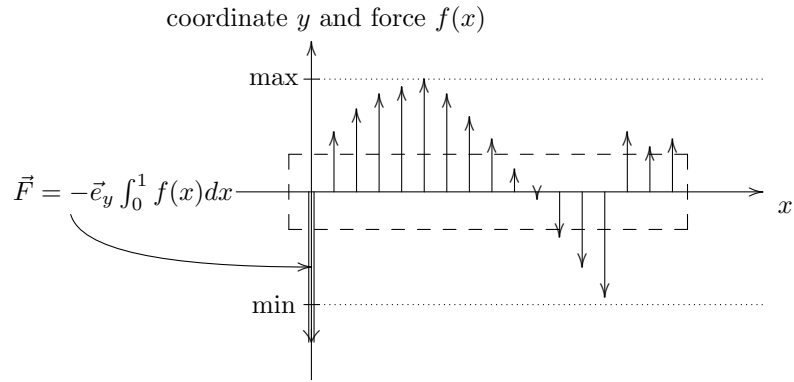

\begin{center}
\beq\xy
(-10,0);(60,0) **@{-}?>*@{>},
+(3,-2)*={x},
(0,-25);(0,20) **@{-}?>*@{>},
+(2,3)*={\textrm{coordinate } y\textrm{ and force }f(x)},
(-3,5);(-3,-5) **@{--}
;(50,-5) **@{--}
;(50,5) **@{--}
;(-3,5) **@{--},
(1,15);(-1,15) **@{-} -(4,0)*={\max}, (60,15) **@{.},
(1,-15);(-1,-15) **@{-} -(4,0)*={\min}, (60,-15) **@{.},
(3,0);(3,8) **@{-}?>*@{>},
(6,0);(6,11) **@{-}?>*@{>},
(9,0);(9,13) **@{-}?>*@{>},
(12,0);(12,14) **@{-}?>*@{>},
(15,0);(15,15) **@{-}?>*@{>},
(18,0);(18,13) **@{-}?>*@{>},
(21,0);(21,10) **@{-}?>*@{>},
(24,0);(24,7) **@{-}?>*@{>},
(27,0);(27,3) **@{-}?>*@{>},
(30,0);(30,-1) **@{-}?>*@{>},
(33,0);(33,-6) **@{-}?>*@{>},
(36,0);(36,-10) **@{-}?>*@{>},
(39,0);(39,-14) **@{-}?>*@{>},
(42,0);(42,8) **@{-}?>*@{>},
(45,0);(45,6) **@{-}?>*@{>},
(48,0);(48,7) **@{-}?>*@{>},
(0,0);(0,-20) **@{=}?>*\dir3{>},
(-25,0)*={\vec{F}=-\vec{e}_y\int_0^1f(x)dx} -(7,3);(0,-10) **\crv{(-30,-10)}?>*@{>}
\endxy\eeq
\end{center}
\caption{Candidate distribution of static friction along a thin rod upon application of the force $\vec{F}$. The dashed lines indicate the contour of the rod as seen from above.}
\label{staticfriction}
\end{figure}

Introducing a Lagrange multiplier as above is a special case of duality theory for linear programs. Hence the following question arises: when formulating convex programming in the context of convex spaces, is there a nice notion of duality that generalizes the classical Karush-Kuhn-Tucker theory? What are appropriate constraint qualifications guaranteeing strong duality?

Since linear programming is a relatively easy optimization problem, one tries to reduce other optimization problems to the linear case. This is done for combinatorial optimization problems in particular, and hence convex spaces might also be of relevance for those.

\begin{ex}[combinatorial optimization]
For us, a combinatorial optimization problem is given by a finite set $X=\left\{x_1,\ldots,x_n\right\}$ (the search space) and a linear subspace
\beq
\mathcal{OF}\subseteq\R^X
\eeq
that is the class of all possible \emph{objective functions}. A particular \emph{instance} of the problem is then given by specifying some $f\in\mathcal{OF}$, and the task is to find the optimal value
\beq
\max_{i=1,\ldots,n}f(x_i)=\mathrm{?}\,.
\eeq
Typically, $n$ is so large that brute-force enumeration of the search space is practically impossible, and therefore one needs to exploit the structure of $\mathcal{OF}$ as the way it lies inside $\R^X$.

For example, let $X$ be the set of all Hamiltonian cycles in a finite graph $G$, and $\mathcal{OF}$ the set of all functions on $X$ which one obtains by assigning a weight to each edge of $G$ and mapping a Hamiltonian cycle to the sum of its edge weights. In this way, one obtains the famous \emph{travelling salesman problem} on $G$.

Since all that matters is how a candidate point $x_i$ behaves under objective functions, we can identify $x_i$ with the evaluation map
\beq
x_i:\mathcal{OF}\ra\R,\quad f\mapsto f(x_i).
\eeq
In this way, $X$ becomes identified with a finite subset of $\R^\mathcal{OF}$. Now consider the polytope
\beq
\mathcal{P}\equiv\mathrm{conv}\left(\left\{x_1,\ldots,x_n\right\}\right)\subseteq\R^\mathcal{OF}.
\eeq
Then each $f\in\mathcal{OF}$ turns into a convex map $f:\mathcal{P}\ra\R$. In practice, one tries to describe $\mathcal{P}$ in terms of linear inequalities, which reduces the combinatorial optimization problem to a linear optimization problem. For example in case of the travelling salesman problem, $\mathcal{P}$ is the \emph{travelling salesman polytope} over $G$.
\end{ex}

It is clear that no list of relevant examples of convex sets could ever be complete. Therefore we simply end this list here by mentioning some particularly severe omissions:\\

\begin{compactitem}
\item Polytopes in general~\cite{Zie} as a certain kind of finitely generated convex spaces.
\item In particular, lattice polytopes and their relation to toric varieties~\cite{Ful}.
\item The geometry of numbers~\cite{Sie} studying integer points (potentially over number fields) in convex subsets of $\R^n$.
\item The Bernstein-Kushnirenko theorem expressing the generic number of non-trivial solutions to a system of polynomial equations in terms of a geometric invariant of a collection of polytopes~\cite{Stu}.
\item The set of Bayesian networks on a fixed directed acyclic graph~\cite{Mad}.
\end{compactitem}

\section{Convex spaces of combinatorial type}
\label{examplescomb}

Now we turn to convex spaces that cannot be embedded as convex subsets of vector spaces. The smallest of these is a convex space structure on a two-element set.

\begin{ex}[two-point convex space]
\label{faceclassifier}
Let $\mathcal{FC}=\{i,f\}$ be a two-element set, and define convex combinations of the two elements as
\beq
\label{twoelemconvspc}
\lambda i+\overline{\lambda}f\equiv\left\{\begin{array}{cc}f&\textrm{if }\lambda=0\\i&\textrm{if }\lambda\neq 0\end{array}\right.
\eeq
This satisfies all the axioms for a convex space.
\end{ex}

Naively, one would deem the previous example pathological. Earlier on in the study of convex spaces, we were also trying to exclude such cases by changing the definition of convex space by requiring $\mathcal{C}$ to be a topological space and the convex combination operations to be continuous. However, we soon found out that example~\ref{twoelemconvspc} is just a special case of a very natural class of convex spaces of combinatorial type, which should not be considered pathological at all. One reason is that $\mathcal{FC}$ from the previous example turns out to be the $\mathcal{F}$ace $\,\mathcal{C}$lassifier for convex spaces~\cite{propclass}, with $f$ representing a $f$ace and $i$ the $i$nterior complement. Another reason is remark~\ref{probposs}.

\begin{defn}
A convex space $\mathcal{C}$ is said to be of \ind{combinatorial type}{combinatorial type} if each function
\beq
(0,1)\lra\mathcal{C},\quad\lambda\mapsto \lambda x+\overline{\lambda}y
\eeq
is constant.
\end{defn}

Then when combining this definition with the axioms\re{unitlaw}--\ref{defparaass}, we see that a convex space of combinatorial type is nothing but a set $\mathcal{C}$ together with a binary operation
\beq
cc_{\frac{1}{2}}:\mathcal{C}\times\mathcal{C}\lra\mathcal{C}
\eeq
which is idempotent, commutative and associative. It is well-known that such an algebraic structure is exactly the same thing as a meet-semilattice, which is a poset $(\mathcal{C},\leq)$ such that each pair of elements has a \ind{meet}{meet}, i.e. a greatest lower bound. In the following, the term \ind{semilattice}{semilattice} always stands for \ind{meet-semilattice}{meet-semilattice}.

We digress briefly to describe the monad and the Lawvere theory underlying semilattices. The monad is a version of the \ind{powerset monad}{powerset monad} (or \ind{Manes monad}{Manes monad}) and is defined over the functor that maps every set to the set of its finite nonempty subsets.

\begin{defn}[the finitary Manes monad]
The finitary Manes monad $\mathscr{M}_\mathrm{fin}\equiv(\mathcal{P}_\mathrm{fin},\varepsilon,\omega)$ is given by the functor
\beq
\mathcal{P}:\mathtt{Set}\ra\mathtt{Set},\quad A\mapsto\mathcal{P}(A)\equiv\{B\subseteq A\:|\:B\neq\emptyset\textrm{ is finite}\}
\eeq
with the obvious action on morphims, the unit natural transformation
\beq
\varepsilon_A:A\ra\mathcal{P}A,\quad x\mapsto\{x\}
\eeq
and the multiplication transformation
\beq
\omega_A:\mathcal{P}\mathcal{P}A\ra\mathcal{P}A,\quad C\mapsto\bigcup_{B\in C}B.
\eeq
\end{defn}

The Lawvere theory of semilattices is the category $\mathtt{FinMultiMap}$ of finite cardinals together with multivalued functions.

We can now see how both the monad and the Lawvere theory underlying convex spaces of combinatorial type are related to $\mathscr{G}_\mathrm{fin}$ and $\mathtt{FinStoMap}$. To this end, consider the semiring $\mathcal{S}_2\equiv\{0,1\}$ with $1+1\equiv 1$. Then the monad $\mathscr{M}_\mathrm{fin}$ originates from $\mathscr{G}_\mathrm{fin}$ by replacing the $\R_{\geq 0}$-coefficients of $\mathscr{G}_\mathrm{fin}$ by $\mathcal{S}_2$-coefficients. In the same way, $\mathtt{FinMultiMap}$ originates from $\mathtt{FinStoMap}$ by making the same change of coefficients: a multivalued function $[m]\ra[n]$ is the same thing as a matrix $M_{n\times m}(\mathcal{S}_2)$ that is ``stochastic'' in the sense that all coefficients sum to $1$.

More formally, changing coefficients along the semiring homomorphism
\beq
\R_{\geq 0}\ra\mathcal{S}_2,\quad\lambda\mapsto\mathrm{sgn}(\lambda)
\eeq
yields a morphism of Lawvere theories $\mathtt{FinStoMap}^\mathrm{op}\ra\mathtt{FinMultiMap}^\mathrm{op}$ and a morphism of monads $\mathscr{G}_\mathrm{fin}\ra\mathscr{M}_\mathrm{fin}$ given by
\beq
\Delta_X\lra\mathcal{P}(X),\quad\sum_{i\textrm{ with }\lambda_i>0}\lambda_ix_i\mapsto\left\{x_1,\ldots,x_n\right\}
\eeq
These morphisms imply that a semilattice naturally carries a convex space structure.

\begin{rem}
\label{probposs}
What does this change of coefficients mean in the information-theoretic interpretation of convex spaces? The answer is that $\mathcal{S}_2$ coefficients only care about qualitative \emph{possibilities}, while $\R_{\geq 0}$ coefficients contain information about quantitative \emph{probabilities}. 
\end{rem}

We now give a few examples of semilattices.

\begin{ex}[free semilattices]
Given a set $X$, the free semilattice over $X$ is given by $\mathcal{C}\equiv\mathcal{P}(X)$ together with the partial order
\beq
A,B\in\mathcal{P}(X):\quad A\leq B\:\Longleftrightarrow A\supseteq B.
\eeq
Then the meet of two finite non-empty subsets of $X$ is given by their union.
\end{ex}

\begin{ex}[possibility measures]
Given a measurable space $(X,\Omega)$, a \ind{possibility measure}{possibility measure} on $(X,\Omega)$ is a map $\mu:\Omega\lra [0,1]$ such that $\mu(\emptyset)=0$, $\mu(X)=1$ and
\beq
\mu\left(\bigcup_{i\in\N}X_i\right)=\sup_{i\in\N}\mu\left(X_i\right)
\eeq
for every countable family of subsets $X_i\in\Omega$.

Intuitively, $\mu$ measures the plausibility an observer assigns to an event. A possibility of $0$ means that the event is impossible. On the other hand, a possibility of $1$ means that the event is totally unsurprising, although it need not occur with absolute certainty.

The set of possibility measures on $(X,\Omega)$ is a semilattice with respect to the ordering
\beq
\mu\leq\mu'\:\Longleftrightarrow \mu(Y)\leq\mu'(Y)\:\:\forall Y\in\Omega.
\eeq
The meet operation is given by
\beq
(\mu_1\land\mu_2)(Y)=\min\left\{\mu_1(Y),\mu_2(Y)\right\}.
\eeq
\end{ex}

\begin{ex}
\label{primes}
Consider $\mathcal{C}=\N$ as a partially ordered set with respect to divisibility:
\beq
x\leq y\:\Longleftrightarrow x|y
\eeq
Then the meet of two natural numbers is given by their greatest common divisor. Hence, $(\N,|)$ is a semilattice which encodes some number-theoretic information.

On the other hand, the decomposition of an integer into its prime factors yields an isomorphism of partially ordered sets $(\N,|)\cong\N^{\times\mathbb{P}}$, where $\P$ denotes the set of prime numbers, and $\N^{\times\mathbb{P}}$ carries the product order. This means that there is nothing to gain from studying the semilattice $(\N,|)$ by itself without any additional structure.
\end{ex}

\section{Convex spaces of mixed type}
\label{examplesmixed}

The above two types of convex spaces should be considered to be extreme cases. In general, a convex space will have a flavor of both the geometrical type and the combinatorial type. For example when starting with a convex space of geometrical type, the following construction will add a combinatorial flavor:

\begin{ex}[adjoining a point at infinity]
\label{adjoininfty}
Let $\mathcal{C}$ be any convex space. Then we define a new convex space as $\mathcal{C}_\infty\equiv\mathcal{C}\cup\{\infty\}$, where the convex combinations are inherited from $\mathcal{C}$ together with, for all points $x\in\mathcal{C}$,
\beq
\lambda\infty+\overline{\lambda}x\equiv\left\{\begin{array}{cc}x&\textrm{for }\lambda=0\\\infty&\textrm{for }\lambda\neq 0\end{array}\right.
\eeq
\end{ex}

There is much more general construction lying behind this example: starting with a semilattice $\mathcal{S}$, we choose a convex space $\mathcal{C}_s$ for each $s\in\mathcal{S}$. The $\mathcal{C}_s$ may be of geometric type, but this is not required. Now we consider the disjoint union
\beq
\mathcal{C}\equiv\bigcup_{s\in\mathcal{S}}\mathcal{C}_s.
\eeq 
Hence, $\mathcal{C}$ is a set over $\mathcal{S}$ with fibers $\mathcal{C}_s$. Furthermore, for every relation $s\leq s'$, we choose a convex map $f_{s,s'}:\mathcal{C}_{s'}\lra\mathcal{C}_s$, such that this data amounts to a functor
\beq
f_{\cdot,\cdot}:\mathcal{S}^{\mathrm{op}}\lra\mathtt{ConvSpc},\quad s\mapsto\mathcal{C}_s,\quad \left(s\leq s'\right)\mapsto f_{s,s'}
\eeq
where the poset $\mathcal{S}$ is considered as a category in the usual way. Now we can define convex combinations on $\mathcal{C}$ as
\beq
\lambda\in(0,1),\,x\in\mathcal{C}_s,\,y\in\mathcal{C}_t:\quad\lambda x+\overline{\lambda}y\equiv\lambda f_{s\land t,s}(x)+\overline{\lambda}f_{s\land t,t}(y)\: \in\mathcal{C}_{s\land t}.
\eeq
Intuitively speaking: for taking a non-trivial convex combination of some point in $\mathcal{C}_s$ and some point in $\mathcal{C}_t$, we have to transport both of them to $\mathcal{C}_{s\land t}$ first and then can take the convex combination there. We denote the resulting convex space by $\mathcal{C}=\mathcal{S}{}_f\ltimes\mathcal{C}_\cdot$.

Example~\ref{adjoininfty} is subsumed by this construction upon setting $\mathcal{S}\equiv\mathcal{FC}$ (from example~\ref{faceclassifier}), $\mathcal{C}_f\equiv\mathcal{C}$ and $\mathcal{C}_i\equiv\{\infty\}$. The map $f_{f,i}:\mathcal{C}\ra\{\infty\}$ is trivially unique.

\begin{ex}[a lottery]
\label{lotteryex}
Suppose we buy a ticket for a lottery. Also suppose that we do not really care about what the prizes are, as long as we win \emph{something}; hence before the results are drawn, we only care about our subjective probability of winning $p\in[0,1]$. But then as soon as we know that we have a winning ticket (i.e. $p=1$), of course we also become interested in what the prize actually is -- the possibilities being, say, an apple $a$ or a banana $b$. Hence in this stage of the process, our subjective state of information is given by an element of $\Delta_{\{a,b\}}$. In total, our possible states of subjective information are given by the convex space
\beq
[0,1)\cup\Delta_{\{a,b\}}
\eeq
where convex combinations within $[0,1)$ or within $\Delta_{\{a,b\}}$ are the ordinary ones, while in addition, for a coefficient $\lambda\in(0,1)$ and a point $p\in[0,1)$,
\beq
\lambda p+\overline{\lambda}\left(\mu a+\overline{\mu}b\right)\equiv\lambda p+\overline{\lambda}.
\eeq
Intuitively speaking, $\Delta_{\{a,b\}}$ acts on $[0,1)$ by convex combinations with $1$. As illustrated in figure~\ref{lotteryquotient}, one can view this convex space as the quotient of $\Delta_{\{p=0,a,b\}}$ where all formal convex combinations with fixed positive coefficient of $p=0$ are identified.

Since $1\notin[0,1)$, this convex space is not of the form $\mathcal{S}{}_f\ltimes\mathcal{C}_\cdot$ for any $\mathcal{S}$ and $\mathcal{C}_\cdot$.
\end{ex}

\begin{figure}
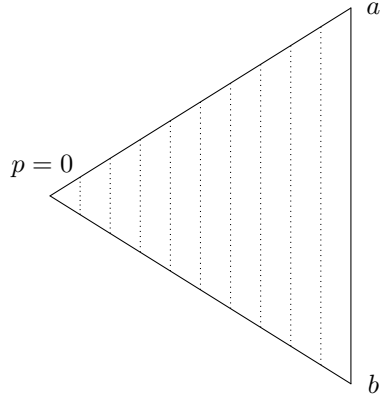

\begin{center}
\beq\xy
(0,0);(40,25) **@{-};
(40,-25) **@{-};
(0,0) **@{-},
(-1,4)*={p=0},
(43,25)*={a},
(43,-25)*={b},
(4,2.5);(4,-2.5) **@{.},
(8,5);(8,-5) **@{.},
(12,7.5);(12,-7.5) **@{.},
(16,10);(16,-10) **@{.},
(20,12.5);(20,-12.5) **@{.},
(24,15);(24,-15) **@{.},
(28,17.5);(28,-17.5) **@{.},
(32,20);(32,-20) **@{.},
(36,22.5);(36,-22.5) **@{.},
\endxy\eeq
\end{center}
\caption{The convex space from example~\ref{lotteryex}. All points on a dotted line are identified, while the points on the line connecting $a$ to $b$ stay distinct.}
\label{lotteryquotient}
\end{figure}

\begin{ex}[convex space of convex sets]
Let $V$ be a real vector space, and take $\mathcal{C}$ to be the set of all convex subsets of $V$:
\beq
\mathcal{C}\equiv\left\{C\subseteq V\:|\:C\textrm{ is convex}\right\}
\eeq
Then convex combinations of two convex subsets $C_1$ and $C_2$ can be defined by
\beq
\lambda C_1+\overline{\lambda}C_2\equiv\left\{\lambda c_1+\overline{\lambda}c_2,\: c_i\in C_i\right\}.
\eeq
Except in the degenerate case $V=0$, this convex space is neither of geometric type nor of combinatorial type. For example when $V=\R$, we can use open and closed intervals to get relations of the form
\beq
\frac{1}{2}(0,1)+\frac{1}{2}[0,1]=(0,1),
\eeq
which cannot possibly hold in a convex space of geometric type. Similar examples abound in higher dimensions.

When considering only those subsets $C\subseteq V$ that are the convex hulls of finitely many points, we obtain the \emph{convex space of polytopes in $V$}. It is unclear whether this convex space is of geometric type.
\end{ex}

\bibliographystyle{halpha}
\bibliography{../convex_spaces}

\vspace{12pt}

\end{document}